\documentclass[11pt]{article}
\usepackage[utf8]{inputenc}
\usepackage{amssymb}
\usepackage{amsmath}
\usepackage{amsfonts}
\usepackage{mathrsfs}
\usepackage{verbatim}
\usepackage{mathtools}
\usepackage{tensor}
\usepackage{amsthm}
\usepackage{enumitem}
\usepackage{hyperref}

\usepackage{authblk}
\usepackage[sort,comma]{natbib}

\usepackage{caption}
\usepackage{booktabs}
\usepackage{siunitx}
\usepackage[T1]{fontenc}
\usepackage{float}

\usepackage{thmtools}
\usepackage{thm-restate}
\usepackage{hyperref}
\usepackage{cleveref}
\usepackage{changepage}

\newcommand{\R}{\mathbb{R}}

\newcommand{\ep}{\varepsilon}
\newcommand{\form}[1]{#1}

\newcommand{\kep}{\frac{\ep-1}{\ep}}
\newcommand{\negkep}{\frac{1-\ep}{\ep}}
\newcommand{\genbeta}[1]{\beta_{#1}(\textbf{q})}
\newcommand{\gengamma}[1]{\gamma_{#1}(\textbf{q})}

\newcommand{\genmu}{\mu(\textbf{q})}
\newcommand{\genv}[1]{v_{#1}(t; \textbf{q})}
\newcommand{\qvec}{\textbf{q}}
\newcommand{\genvprime}[1]{v'_{#1}(t; \qvec)}

\newcommand{\lsup}[1]{\limsup\limits_{{#1}\rightarrow\infty}}
\newcommand{\linf}[1]{\liminf\limits_{{#1}\rightarrow\infty}}
\newcommand{\result}{\mu(\qvec)^\negkep}

\newcommand{\optp}[1]{p_{#1}^*(t;\qvec)}
\newcommand{\single}[1]{\delta_{#1}}

\newtheorem{theorem}{Theorem}
\newtheorem{corollary}[theorem]{Corollary}

\title{Dynamic Pricing with Variable Order Sizes for a Model with Constant Demand Elasticity\footnote{© 2018. This manuscript version is made available under the CC-BY-NC-ND 4.0 license \url{http://creativecommons.org/licenses/by-nc-nd/4.0/}}}
\author{
  Nyles Breecher\footnote{Corresponding author, breecher@uwm.edu, University of Wisconsin Milwaukee, Dept of Mathematical Sciences, PO Box 413, Milwaukee, WI 53201, USA}, Richard Stockbridge\footnote{stockbri@uwm.edu, University of Wisconsin Milwaukee, Dept of Mathematical Sciences, PO Box 413, Milwaukee, WI 53201, USA} \\
  \textit{University of Wisconsin Milwaukee}
}
\date{\today}

\begin{document}

\maketitle

\subsection*{Keywords}

Dynamic Pricing, Constant Demand Elasticity, Optimal Stochastic Control, Variable Order Sizes

\subsection*{Abstract}

In this paper we investigate a dynamic pricing model for constant demand elasticity where customers have a probability distribution on the number of items they order. This is a generalization from standard models which restrict customers to buy only one item at a time. For the generalized model, we first obtain a closed form expression for the optimal expected revenue and optimal pricing strategy. This expression involves a recursively defined term for which we investigate the behavior. We call comparable models those which have the same demand, which is the customer arrival rate times the average order size. In fact, the average order size plays an important role for results for the generalized model. An important result we show is that comparable models have the same asymptotic pricing behavior. Numerical results also show that comparable models are relatively close even for low inventory levels. Lastly, we prove that the relative difference between comparable models is governed not by the customer arrival rate, but solely by their order size distributions.

\section{Introduction}

\subsection{Overview}

Dynamic pricing concerns sellers who attempt to maximize their profits by choosing a pricing strategy based on market conditions. We will consider the problem of finding an optimal pricing strategy based on a limited inventory and limited time to sell that inventory. The most common example of this type of problem occurs with airline seat pricing: the price of seats can change dramatically up until the flight, but seats cannot be sold after that time. Other examples include fashion goods and hotel rooms. Standard models allow customers to purchase one item at a time; our generalization to this model is to allow customers to order multiple items at once. In particular, we look at a model where the demand elasticity is constant. Throughout this paper we will refer to the "standard model" as a model where customers order only a single item, and the "generalized model" to refer to our model where there is a probability distribution on the amount of items customers order.

The model we use is a random time change of a compound Poisson process with rate $\lambda(p,t)$ indicating the arrival rate of customers willing to pay $p$ at time $t$. The random time change portion comes from the fact that $\lambda$ depends on $p$, which varies based on any particular realization of the model. Here we make a clarification on the term "demand." Demand refers to the amount of items desired per unit of time. When customers order one item, this term is synonymous with the customer arrival rate; however, when customers have a probability distribution on the amount of items they order, demand is equal to $\lambda(p,t)\mu$, where $\mu$ is the average order size. There are many dynamic pricing models which investigate various forms for the customer arrival function $\lambda$. We will consider the specific case where demand elasticity is constant. Such demand commonly appears in economics, and provides a tractable example for variable order sizes. In this case, $\lambda(p,t)$ takes form $\lambda(p,t)=a(t)p^{-\ep}$ where $a(t)$ is a scaling factor for demand over time and $\ep$ is the demand elasticity.

 There are two important questions when generalizing to variable order sizes that we wish to address. First, how do the results from the standard model generalize? We observe that the average order size $\mu$ plays an important result in these generalized results, a term which is not observed for the standard model since $\mu=1$. The second question is how does the standard model compare to the generalized model? We will refer to "comparable models" as those which have the same demand.

The first result we obtain is a closed form expression for the optimal expected revenue and the optimal pricing strategy for the variable order size model. That said, the expression involves a recursively defined term which requires further examination. Analyzing this term proves to be more difficult than it might appear at first glance. In particular, it requires different proof techniques compared to those for the standard model. After analysis, we prove that a variable order size model and its comparable standard model are asymptotically equivalent as the inventory approaches infinity. Additionally, we provide some numerical examples which show that these comparable models are similar even for small inventory values. As a last--and surprising--result, we find that the relative difference between comparable models are unaffected by the customer arrival rate $\lambda(p,t)$, but can be determined by their order size distributions.

\subsection{Literature Review}

In this section we discuss papers which are most directly related to our work.

General theory for dynamic pricing models in continuous time has been well elucidated in \cite{Gallego}. This paper focuses on how the two main properties of dynamic pricing models, a limited stock and time to sell it, influence the model. It explores the general theory around finding optimal revenue and pricing strategies and finds a closed form expression for the optimal expected revenue and pricing strategy for an exponential customer arrival rate function $\lambda=e^p$. Some generalizations to the basic problem are also included, such as a brief description of how to work with a compound Poisson process, but no specific $\lambda$ functions are examined with this lens.

A wider overview of dynamic pricing models is provided by \cite{Talluri}. Of particular note is the comparison between dynamic pricing and inventory control.

Later, the article by \cite{Mcafee} applies the work by Gaellgo and van Ryzin to the case of constant demand elasticity $\ep$, a situation which requires the specific demand function $\lambda=p^{-\ep}$. They consider constant demand elasticity since it has practical applications to economics. Much of our work for the variable order size model in this paper draws from the work present in this paper; however, generalizing for variable order sizes requires different proof techniques. Our initial results parallel theirs, but with the introduction of the average order size $\mu$ in several places.

The paper by \cite{Monahan} also considers a dynamic pricing model with constant demand elasticity; however, its model evolves in discrete time. Under this setup customers having variable order sizes not not arise in the model, as all that matters is the average amount of sales during each time period.

The paper by \cite{Helmes} considers a different extension to the dynamic pricing problem with constant demand elasticity by including dynamic advertising.

Other authors have explored models where customers specify the amount of items they wish to order. For example, \cite{Sawaki} uses a semi-Markov decision process for modeling.

\subsection{Dynamic Programming Formulation}

In this section we define the notation and describe the assumptions we use for our modeling.

Our model is a random time change of a compound Poisson process with customer arrival rate $\lambda(p,t)$ dependent on price $p$ and current time $0\leq t\leq T$, where $T$ is the maximum sale time. Each customer who arrives has a probability distribution $\qvec$ on the amount of items they order, and can order a maximum of $M$ items. A finite maximum order size helps computation, but matters from a practical sense: one cannot sell more than their maximum possible inventory. Write $\qvec=(q_1,q_2,\ldots,q_M)$ to indicate customers buy $i$ items with probability $q_i$ (thus $\sum_{i=1}^M q_i=1$). Define the average order size $\genmu=\sum_{i=1}^M q_i$. Thus demand is given by $q(p,t)=\lambda(p,t)\genmu$, not $\lambda(p,t)$, a distinction which is necessary when working with variable order sizes.

A demand function is said to have constant elasticity $\ep$ if
\begin{equation}
\ep=-\frac{p\left(\frac{dq}{dp}\right)}{q(p)}=-\frac{p\left(\frac{d\lambda}{dp}\right)}{\lambda(p,t)},\label{eq_ep}
\end{equation}
which has solution $\lambda(p,t)=a(t)p^{-\ep}$ for some time-dependent function $a(t)$. $a(t)$ can be thought of as the customer arrival rate for $p=1$. Note that by (\ref{eq_ep}), comparable models have equal demand elasticities $\ep$.

Consider now the problem of a company trying to maximize its revenue rate $r$. In light of constant elasticity $\ep$ and marginal cost $c>0$, the maximum revenue rate is given by
$$r^*=\max_p (p-c)q(p)=\max_p (p-c)a(t)p^{-\ep}\genmu.$$
This has no practical solution if $\ep\leq 1$, as the company would need to sell at an infinitely high price. Thus we assume $\ep>1$.

To continue developing the model, we follow the method as outlined by Gallego and van Ryzin (1994). Let $\genv{n}$ be the optimal expected revenue for a seller with $n$ items to sell at time $0\leq t\leq T$. Thus we have the natural constraints
$$\genv{0}=0 \quad \text{and} \quad v_n(T; \genmu)=0,$$
which show that no revenue can be made if there is no inventory to sell or there is no time left to sell.

We will heuristically derive the Hamilton-Jacobi equations for $v'(t; \qvec)$ following \cite{Gallego}. These equations will give us a condition in order to solve for $v(n;\qvec)$. Consider the optimal expected revenue for $n>0$ over a time span of $\delta t$. During this time span a customer arrives with probability $\lambda(p,t)\delta t$, and they order $i$ items with probability $q_i$. Therefore, using the Principle of Optimality, we can say
\begin{equation}
    v_n(t; \qvec)=\sup\limits_p \left[\underbrace{(1-\lambda(p, t)\delta t)v_n(t+\delta t; \qvec)}_{\text{From selling no items}}+\sum_{i=1}^M \underbrace{q_i\lambda(p, t)\delta t( ip+v_{n-i}(t+\delta t; \qvec))}_{\text{From selling $i$ items for $p$ each}}\right]. \label{eq_v}
\end{equation} 

In order to use the above equation, we need to define some base cases. The summation requires $M$ terms, meaning $\genv{n}$ needs to be defined for for $n\leq 0$, but what does a ``negative inventory'' mean? Depending on how these values are defined, multiple interpretations can arise. If $\genv{n}=0$ for $n\leq 0$, this equates to overselling. This might be reasonable, such as in the case of airline flights, which often overbook. If $\genv{n}=pn$ for $n\leq 0$, this negates any profit earned from overselling, which equates to customers buying only the available inventory if they would otherwise buy more items than are available. Other values of $\genv{n}$ for $n\leq 0$ can be chosen to appropriately model behavior as inventory becomes small. For simplicity, we will allow the case of overselling, i.e. $\genv{n}=0$ for $n\leq 0$, although our results still hold for different definitions of the base terms. Additionally, it's worth noting that the overall optimal expected revenue will not be affected much if the maximum inventory is sufficiently large.

From Equation (\ref{eq_v}) we finish development of the the Hamilton-Jacobi equations. To find $v_n'(t; \qvec)$. Rearranging (\ref{eq_v}) we get
\begin{equation*}
\begin{split}\frac{v_n(t+\delta t; \qvec) -v_n(t; \qvec)}{\delta t} = & - \sup\limits_p \lambda(p, t)\bigg[-v_n(t+\delta t; \qvec) +\sum_{i=1}^M q_i( ip+v_{n-i}(t+\delta t; \qvec))\bigg],
\end{split}
\end{equation*}
now take the limit as $\delta t\rightarrow 0$. Note that the interchange of the supremum and limit has not been justified. For full details, see \cite{Bremaud}. This yields
$$v_n'(t; \qvec)=-\sup\limits_p \lambda(p, t)\left[- v_n(t; \qvec)+\sum_{i=1}^M q_i(ip+v_{n-i}(t; \qvec))\right].$$
Then substitute $\lambda(p,t)=a(t)p^{-\ep}$ and rewrite,
\begin{align}
    v_n'(t; \qvec) & = -\sup\limits_p a(t)p^{-\ep}\left[-v_n(t; \qvec)+\sum_{i=1}^M q_i ( ip+v_{n-i}(t; \qvec))\right] \nonumber \\
    & = -\sup\limits_p a(t)p^{-\ep}\left[-v_n(t; \qvec)+\sum_{i=1}^M q_i i p +\sum_{i=1}^M q_i v_{n-i}(t; \qvec)\right] \nonumber \\
    & = -\sup\limits_p a(t)p^{-\ep}\left[\genmu p - \left(v_n(t; \qvec) - \sum_{i=1}^M q_i v_{n-i}(t; \qvec)\right)\right] \nonumber \\
    & =: -\sup\limits_p a(t)p^{-\ep}\left[\genmu p - C\right]. \label{eq_v_bellman}
\end{align}
Find the supremum by setting the $p$-derivative of the right-hand side to 0,
$$0=a(t)\genmu(-\ep+1)p^{-\ep}+a(t)C\ep p^{-\ep-1}=a(t)p^{-\ep-1}(\genmu(1-\ep)p+\ep C),$$
so the price $p^*$ which obtains the maximum is
\begin{equation}
    \form{p}^*=\frac{\ep}{\ep-1}\genmu^{-1}C=\frac{\ep}{\ep-1}\genmu^{-1}\left(v_n(t; \qvec) - \sum_{i=1}^M q_i v_{n-i}(t; \qvec)\right).\label{eq_maximal_p}
\end{equation}
Substituting this value for $p$ into Equation (\ref{eq_v_bellman}) yields
\begin{align}
    \genvprime{n} & = -a(t)\left(\frac{\ep}{\ep-1}\genmu^{-1}C\right)^{-\ep}\left(\frac{\ep}{\ep-1}C-C\right) \nonumber \\
    & = -a(t)\genmu^{\ep}\left(\frac{\ep-1}{\ep}\right)^{\ep}C^{1-\ep}\left(\frac{\ep}{\ep-1}-1\right) \nonumber \\
    & = -a(t)\genmu^{\ep}\frac{(\ep-1)^{\ep-1}}{\ep^\ep}C^{1-\ep} \nonumber \\
    & = -a(t)\genmu^{\ep}\frac{(\ep-1)^{\ep-1}}{\ep^\ep}\left(v_n(t; \qvec) - \sum_{i=1}^M q_i v_{n-i}(t; \qvec)\right)^{1-\ep}. \label{vn-prime}
\end{align}
Thus we have a formula for $\genvprime{n}$, which by the Principle of Optimality is a necessary and sufficient condition for $\genv{n}$.

\section{Results}

\subsection{Overview}

The first result presents a closed form for the optimal expected revenue and optimal pricing strategy. However, it involves an recursively computable term which has a closed form expression, but is difficult to understand from the formula alone. The results after the first delve into the properties of this term, culminating in a theorem which says that comparable models have the same inventory-asymptotic behavior. (Recall that we define two models as comparable if their demands $\lambda \mu$ are equal.) Or in other words, a model allowing variable order sizes may be approximated with a model allowing a single order size.

As a stepping stone for our formulas, we need to define the non-negative sequence $(\beta_n(\textbf{q}))_n$ for $n\leq 0$ by $\genbeta{n}=0$ and for $n>0$ so that $\genbeta{n}$ is satisfies the recursive equation:
\begin{equation}
    \kep=\genbeta{n}^{\frac{1}{\ep-1}}\left(\genbeta{n}-\sum_{i=1}^Mq_i \genbeta{n-i}\right).
    \label{eq_genbeta}
\end{equation}
As one can see, this sequence captures some of the structure in (\ref{vn-prime}). Despite a cumbersome equation, each $\genbeta{n}$ can be computed numerically. Another term which is needed is $A(t)=\int_t^Ta(s)ds$, and can be thought of as the expected number of future sales at a price of 1.

\subsection{Analytic Results}

 The first theorem relates the optimal expected revenue $\genv{n}$ and the optimal expected price $p^*_n(t; \qvec)$ to the term $\genbeta{n}$, yielding a closed form expression for $\genv{n}$ and $p^*_n(t; \qvec)$. Note that $\genmu$ appears in the formula, something which is not noticed when customers can only order one item at a time.
 
\begin{theorem}
The optimal expected revenue with $\lambda(p,t)=a(t)p^{-\ep}$ and variable order sizes is given by
\begin{equation}
    v_n(t; \textbf{q}) = \genmu\genbeta{n} A(t)^{1/\ep} \label{eq_optimal_v}
\end{equation}
for all integers $n$. Furthermore, the optimal price is given by
$$p_n^*(t; \textbf{q})=\genbeta{n}^{-1/(\ep-1)}A(t)^{1/\ep}$$\label{thm_optimal_expected_revenue}
\end{theorem}

\begin{proof}
We show through induction that
\begin{equation}
    \genv{n} = \genmu\genbeta{n}A(t)^{1/\ep}. \label{eq-induction-assumption}
\end{equation}
For $n\leq 0$, by definition, $\genbeta{n}=0$ and $\genv{n}=0$, showing that (\ref{eq-induction-assumption}) holds. Assume Equation (\ref{eq-induction-assumption}) holds up to $n$. Recall Equation (\ref{vn-prime}),
$$\genvprime{n+1}= -a(t)\genmu^{\ep}\frac{(\ep-1)^{\ep-1}}{\ep^\ep}\left(v_{n+1}(t; \qvec) - \sum_{i=1}^M q_i v_{{n+1}-i}(t; \qvec))\right)^{1-\ep}.$$
Suppress $t$ and $\textbf{q}$ dependencies for space, and apply the induction assumption to get
\begin{equation}
v_{n+1}'= -a\mu^{\ep}\frac{(\ep-1)^{\ep-1}}{\ep^\ep}\left(v_{n+1} - \sum_{i=1}^M q_i \mu\beta_{n+1-i}A^{1/\ep}\right)^{1-\ep}.   \label{eq_v_nplus1} 
\end{equation}
To verify the induction assumption for $n+1$, we need only verify the desired result holds when substituted into (8). The left-hand side of (8) is
\begin{align*}
    v'_{n+1} & = \mu\beta_{n+1}\frac{1}{\ep}A^{(1/\ep)-1}\frac{dA}{dt} \\
    & = \mu\beta_{n+1}\frac{1}{\ep}A^{(1-\ep)/\ep}(-a) \\
    & = -a\mu A^{(1-\ep)/\ep}\frac{1}{\ep}\beta_{n+1},
\end{align*}
and the right-hand side is
\begin{align*}
    & - a\mu^{\ep}\frac{(\ep-1)^{\ep-1}}{\ep^\ep}\left(\mu\beta_{n+1}A^{1/\ep} - \sum_{i=1}^M q_i \mu\beta_{n+1-i}A^{1/\ep}\right)^{1-\ep} \\
     = & - a\mu\frac{(\ep-1)^{\ep-1}}{\ep^\ep}A^{(1-\ep)/\ep}\left(\beta_{n+1} - \sum_{i=1}^M q_i \beta_{n+1-i}\right)^{1-\ep} \\
     = &  -a\mu A^{(1-\ep)/\ep}\frac{1}{\ep} \left(\frac{\ep-1}{\ep}\right)^{\ep-1}\left(\beta_{n+1} - \sum_{i=1}^M q_i \beta_{n+1-i}\right)^{1-\ep} \\
     = & -a\mu A^{(1-\ep)/\ep}\frac{1}{\ep}\beta_{n+1},
\end{align*}
showing that the left- and right-hand sides of (\ref{eq_v_nplus1}) are equal and verifying the induction assumption for $n+1$. Therefore for all $n$,
$$\genv{n}=\genmu\genbeta{n} A(t)^{1/\ep}.$$

Furthermore, we can substitute this into Equation (\ref{eq_maximal_p}) to get

\begin{equation}
    \form{p}^*=\frac{\ep}{\ep-1}\genmu^{-1}C=\frac{\ep}{\ep-1}\genmu^{-1}\left(v_n(t; \qvec) - \sum_{i=1}^M q_i v_{n-i}(t; \qvec)\right)
\end{equation}
\begin{align*}
    p^* & = \frac{\ep}{\ep-1}\genmu^{-1}\left(v_n(t; \qvec) - \sum_{i=1}^M q_i v_{n-i}(t; \qvec)\right) \\
    & = \frac{\ep}{\ep-1}\genmu^{-1}\left(\genmu\genbeta{n}A(t)^{1/\ep}- \sum_{i=1}^M q_i \genmu\genbeta{n-i}A(t)^{1/\ep}\right) \\
    & = A(t)^{1/\ep}\frac{\ep}{\ep-1}\left(\genbeta{n}- \sum_{i=1}^M q_i \genbeta{n-i}\right) \\
    & = A(t)^{1/\ep}\left(\left(\frac{\ep-1}{\ep}\right)^{\ep-1}\left(\genbeta{n}- \sum_{i=1}^M q_i \genbeta{n-i}\right)^{1-\ep}\right)^{-1/(\ep-1)} \\
    & = \genbeta{n}^{-1/(\ep-1)}A(t)^{1/\ep}.
\end{align*}
Finishing the proof.
\end{proof}

The previous proof shows relationships between $\genv{n}$ and $\genbeta{n}$ and $p^*_n(t;\qvec)$ and $\genbeta{n}$, but to understand these equations we need to understand $\genbeta{n}$ better. The first theorem about $\genbeta{n}$ shows that it is non-decreasing in $n$. With Theorem \ref{thm_optimal_expected_revenue} in mind, this means that as $n$ increases, $v_n(t;\qvec)$ never decreases, which matches with our intuition. It also means that $p_n^*(t;\qvec)$ never increases as $n$ increases, which also matches out intuition. After the following theorem we look at the long term behavior of $\genbeta{n}$ to get a sense how much it changes relative to $n$.

\begin{theorem}
$\genbeta{n}$ is a non-decreasing sequence in $n$. \label{thm_nondecreasing}
\end{theorem}

\begin{proof}
    For $n\geq 0$, $\genbeta{-n}=0$. Proceed by induction and assume $\genbeta{k-1}\leq\genbeta{k}$ for all $k$ from 0 to $n-1$.
    From the definition in (\ref{eq_genbeta}), we have $$\kep=\genbeta{n}^{\frac{1}{\ep-1}}\left(\genbeta{n}-\sum_{i=1}^Mq_i \genbeta{n-i}\right),$$
    or multiplied out as
    $$\kep=\genbeta{n}^{\frac{\ep}{\ep-1}}-\sum_{i=1}^Mq_i \genbeta{n-i}\genbeta{n}^{\frac{1}{\ep-1}}. \label{eq_solve_for_beta_n}$$
    Note that to solve for $\genbeta{n}$ we do so iteratively, meaning that $\genbeta{m}$ for any $m<n$ would already be known and treated as constants. Thus we can treat the right hand side of the above equation as a function of the variable $\genbeta{n}$, so define
    \begin{equation*}
        f(x):=x^{\frac{\ep}{\ep-1}}-\sum_{i=1}^Mq_i \genbeta{n-i}x^{\frac{1}{\ep-1}},
    \end{equation*}
    which capture these features. Solving for $\genbeta{n}$ would be equivalent to solving for $x$ when $\kep=f(x)$. We also have
    \begin{equation*}
        g(x):=x^{\frac{\ep}{\ep-1}}-\sum_{i=1}^Mq_i \genbeta{(n-1)-i}x^{\frac{1}{\ep-1}},
    \end{equation*}
    which we would use to find $\genbeta{n-1}$ instead of $\genbeta{n}$.
    
    By the induction hypothesis, the coefficients in $f(x)$ are less than or equal to those in $g(x)$. Therefore, for positive $x$, $f(x)\leq g(x)$. Thus we know the positive solution to $f(x)=\kep$ must be greater than or equal to the solution to $g(x)=\kep$, or in other words, $\genbeta{n}\geq\genbeta{n-1}$. Note also that
    \begin{align*}
        f(x^{\ep-1}) & :=x^{\ep}-\sum_{i=1}^Mq_i \genbeta{n-i}x \\
        g(x^{\ep-1}) & :=x^\ep-\sum_{i=1}^Mq_i \genbeta{(n-1)-i}x,
    \end{align*}
    by which we can see that $f(x)$ and $g(x)$ start at 0, then decrease until some $x$, then increase from then on. Therefore the equations $\kep=f(x)$ and $\kep=g(x)$ have only one solution each for $x>0$. This completes the induction proof. Therefore $\genbeta{n}$ is a non-decreasing sequence in $n$.
\end{proof}

Next we wish to find the long-term behavior of $\genbeta{n}$ in order to understand how it changes relative to $n$. We find, in the next theorem, that
\begin{equation}
\lim\limits_{n\rightarrow\infty}\frac{\beta_n(\textbf{q})}{n^{\kep}}=\mu(\textbf{q})^{\negkep}.\label{eq_temp_result}
\end{equation}
To make the above ratio easier to work with, define $\gengamma{n}=0$ for  $n\leq 0$ and 
$$\gengamma{n}=\frac{\genbeta{n}}{n^{\kep}}$$
for $n>0$. When customers order only one item, $\genmu$ is actually a monotone increasing term. It also bounded above by the limiting value $\result$. However, when generalizing to variable order sizes, $\genmu$ is no longer monotone, and also isn't bounded above by $\result$ for all values of $n$. This creates two problems in proving the limit in (\ref{eq_temp_result}), and is why we have different proof techniques than those used in McAfee and te Velde \cite{Mcafee}. To prove this limit, we instead look to the liminf and limsup. A pair of technical lemmas describe some qualities of $\limsup\limits_{n\rightarrow\infty}\genbeta{n}$ and $\liminf\limits_{n\rightarrow\infty}\genbeta{n}$, which will combine nicely when we finally reach the proof for (\ref{eq_temp_result}). Recall $M$ is the maximum order size.

\begin{restatable}{lemma}{lemgammabounds}\label{lem_gamma_bounds}
    \begin{enumerate}[label={(\alph*)}]
        \item $ \liminf\limits_{n\rightarrow\infty} \gengamma{n}>0.$ \label{lem_gamma_bounds_a}
        \item If there exists a strictly increasing sequence $(N_k)_k\subset \mathbb{N}$ such that $\gengamma{N_k}=\min\limits_{0\leq i\leq M} \gengamma{N_k-i}$ for all $k$, then $\liminf\limits_{n\rightarrow\infty}\gengamma{n}\geq \result$.\label{lem_gamma_bounds_b}
        \item If there exists an $N\geq M$ such that $\gengamma{N}=\max\limits_{0\leq i\leq M}\gengamma{N-i}$, then $\limsup\limits_{n\rightarrow\infty}\gengamma{n}\leq\result$. \label{lem_gamma_bounds_c}
    \end{enumerate}
\end{restatable}

The other lemma shows important relationships between $\limsup \gengamma{n}$ and $\liminf\gengamma{n}$. The proof of this lemma is quite detailed, but provides the bulk of the work towards proving Theorem \ref{thm_big_result}.

\begin{restatable}{lemma}{lemrelations} \label{lem_limsup_liminf_relations}
    \begin{enumerate}[label={(\alph*)}]
        \item
        \begin{equation*}
            \frac{1}{\limsup\limits_{n\rightarrow\infty}\gengamma{n}^{\frac{1}{\ep-1}}}\leq \genmu\liminf\limits_{n\rightarrow\infty}\gengamma{n}.
        \end{equation*} \label{lem_limsup_liminf_relations_a}
        \item
        \begin{equation*}
            \frac{1}{\liminf\limits_{n\rightarrow\infty}\gengamma{n}^{\frac{1}{\ep-1}}}\geq \genmu\limsup\limits_{n\rightarrow\infty}\gengamma{n}.
        \end{equation*} \label{lem_limsup_liminf_relations_b}
    \end{enumerate}
\end{restatable}

We now present the main result of the paper, which shows the long term behavior of $\genbeta{n}$. In McAfee and te Velde (2008), they obtain this type of result as well, but their model orders of size 1. The proof techniques here are quite different from that paper, highlighting that the generalization to variable order sizes is non-trivial. Note that again the average order size $\genmu$ appears, this time determining a scaling factor for the limit.

\begin{theorem}
    For any choice of $\ep>1$ and order size distribution $\qvec$,
    $$\lim\limits_{n\rightarrow\infty}\frac{\beta_n(\textbf{q})}{n^{\kep}}=\mu(\textbf{q})^{\negkep}.$$ \label{thm_big_result}
\end{theorem}

\begin{proof}
    Substitute $\gengamma{n}n^{\kep}$ for $\genbeta{n}$ into (\ref{eq_genbeta}) to get
    \begin{align}
        \kep=&\gengamma{n}^{\frac{1}{\ep-1}}n^{1/\ep}(\gengamma{n} n^{\kep}-\sum_{i=1}^Mq_i\gengamma{n-i}(n-i)^{\kep}) \nonumber\\
        =&\gengamma{n}^{\frac{1}{\ep-1}}n\left(\gengamma{n}-\sum_{i=1}^Mq_i\gengamma{n-i}\left(\frac{n-i}{n}\right)^{\kep}\right) \label{eq_gengamma}.
    \end{align}
    
    The argument will be split into three cases, based on the following two properties:
    \begin{description}
        \item[Property 1:] Suppose there exists an increasing sequence $(N_k)_k$ such that $\gengamma{N_k}=\min\limits_{0\leq i\leq M} \gengamma{N_k-i}$ for all $k$. 
        \item[Property 2:] Suppose there exists an $N\geq M$ such that $\gengamma{N}=\max\limits_{0\leq i\leq M}\gengamma{N-i}$.
    \end{description}
    
    \textbf{Case 1:} Suppose Property 1 holds. By Lemma $\hyperref[lem_gamma_bounds]{\ref*{lem_gamma_bounds}\ref*{lem_gamma_bounds_b}}$,
    $$\liminf\limits_{n\rightarrow\infty}\gengamma{n}\geq\result .$$
    Using this with Lemma $\hyperref[lem_limsup_liminf_relations]{\ref*{lem_limsup_liminf_relations}\ref*{lem_limsup_liminf_relations_b}}$, we get
    \begin{equation*}
        \frac{1}{(\result)^{\frac{1}{\ep-1}}}\geq \frac{1}{\liminf\limits_{n\rightarrow\infty}\gengamma{n}^{\frac{1}{\ep-1}}}\geq \genmu\limsup\limits_{n\rightarrow\infty}\gengamma{n},
    \end{equation*}
    and simplifying $\genmu$ terms yields
    $$\result\geq \limsup\limits_{n\rightarrow\infty}\gengamma{n}.$$
    Therefore
    $$\result \geq \limsup\limits_{n\rightarrow\infty}\gengamma{n}\geq \liminf\limits_{n\rightarrow\infty}\gengamma{n} \geq \result,$$
    and equality holds throughout.
    
    \textbf{Case 2:} Suppose Property 2 holds. By Lemma $\hyperref[lem_gamma_bounds_c]{\ref*{lem_gamma_bounds}\ref*{lem_gamma_bounds_c}}$,
    $$\limsup\limits_{n\rightarrow\infty}\gengamma{n}\leq\result .$$
    Using this with lemma $\hyperref[lem_limsup_liminf_relations]{\ref*{lem_limsup_liminf_relations}\ref*{lem_limsup_liminf_relations_a}}$ we get
     \begin{equation*}
        \frac{1}{(\result)^{\frac{1}{\ep-1}}}\leq \frac{1}{\limsup\limits_{n\rightarrow\infty}\gengamma{n}^{\frac{1}{\ep-1}}}\leq \genmu\liminf\limits_{n\rightarrow\infty}\gengamma{n},
    \end{equation*}
    and simplifying the $\genmu$ terms yields
    $$\result\leq \liminf\limits_{n\rightarrow\infty}\gengamma{n}.$$
    Therefore
    $$\result \geq \limsup\limits_{n\rightarrow\infty}\gengamma{n}\geq \liminf\limits_{n\rightarrow\infty}\gengamma{n} \geq \result,$$
    and equality holds throughout.
    
    \textbf{Case 3: } Suppose Properties 1 and 2 are both false. Since Property 1 is false, there exists an $N_1>M$ such that for all $n\geq N_1$, $\gengamma{n}\neq \min\limits_{0\leq i\leq M}\gengamma{n-i}$. Let $(N_k)_k$ be a strictly increasing sequence starting with $N_1$. Now let $N_k-M\leq a_k\leq N_k$ be such that $\gengamma{a_k}=\min\limits_{0\leq i\leq M}\gengamma{N_k-i}$. Since Property 1 is false, we have that for all $n>a_k$
    \begin{equation*}
        \gengamma{a_k}<\gengamma{n},
    \end{equation*}
    and thus
    \begin{equation}
        \gengamma{a_k}\leq\liminf\limits_{n\rightarrow\infty}\gengamma{n}.\label{eq_liminf_bound}
    \end{equation}
    Let $N_k-M\leq b_k\leq N_k$ be such that $\gengamma{b_k}=\max\limits_{0\leq i \leq M}\gengamma{N_k-i}$. Since Property 2 is false we have for all $n> b_k$ that
    \begin{equation*}
        \gengamma{n}<\gengamma{b_k},
    \end{equation*}
    and thus
    \begin{equation}
        \limsup\limits_{n\rightarrow\infty}\gengamma{n}\leq \gengamma{b_k}. \label{eq_limsup_bound}
    \end{equation}
    Equations (\ref{eq_liminf_bound}) and (\ref{eq_limsup_bound}) together imply that
    \begin{equation}
        \limsup\limits_{n\rightarrow\infty}\gengamma{n}-\liminf\limits_{n\rightarrow\infty}\gengamma{n}\leq \gengamma{b_k}-\gengamma{a_k}. \label{eq_limsup-liminf_leq_bk-ak}
    \end{equation}
    Thus if we find a bound $\gengamma{b_k}-\gengamma{a_k}$, we have a bound for $\limsup\limits_{n\rightarrow\infty}\gengamma{n}-\liminf\limits_{n\rightarrow\infty}\gengamma{n}$. Further if such a bound goes to 0, we will obtain the useful result
    $$\limsup\limits_{n\rightarrow\infty}\gengamma{n}=\liminf\limits_{n\rightarrow\infty}\gengamma{n}=\lim\limits_{n\rightarrow\infty}\gengamma{n}.$$
    
    To help towards finding a bound, note Equation (\ref{eq_beta_step_difference}) from the proof of Lemma \ref{lem_limsup_liminf_relations} in the appendix:
    $$\genbeta{n}-\genbeta{n-1}\leq \left(\kep\right)\left(\frac{1}{\genbeta{n}^{1/(\ep-1)}}\right).$$
    Since $\gengamma{n}$ is a function of $\genbeta{n}$, we wish to use the above inequality by obtaining some other inequality involving $\genbeta{n}-\genbeta{n-1}$ terms. Suppose that $a_k\leq b_k$. Then
    \begin{align}
        \gengamma{b_k}-\gengamma{a_k} & = \frac{\genbeta{b_k}}{b_k^\kep}-\frac{\genbeta{a_k}}{a_k^\kep} \nonumber \\
        & \leq \frac{\genbeta{b_k}}{b_k^\kep}-\frac{\genbeta{a_k}}{b_k^\kep} \nonumber \\
        & = \frac{1}{b_k^\kep}\left(\genbeta{b_k}-\genbeta{a_k}\right) \nonumber \\
        & = \frac{1}{N_k^\kep}\sum_{i=a_k+1}^{b_k}\genbeta{i}-\genbeta{i-1} \nonumber \\
        & \leq \frac{1}{b_k^\kep}\sum_{i=a_k+1}^{b_k}\left(\kep\right)\left(\frac{1}{\genbeta{i}^{1/(\ep-1)}}\right) \nonumber \\
        & \leq \frac{1}{b_k^\kep}\sum_{i=a_k+1}^{b_k}\left(\kep\right)\left(\frac{1}{\genbeta{a_k+1}^{1/(\ep-1)}}\right) \nonumber \\
        & \leq \frac{1}{b_k^\kep}\left(\frac{M}{\genbeta{a_k+1}^{1/(\ep-1)}}\right), \label{eq_bound_for_bk-ak}
    \end{align}
    where the second to last inequality is justified because $\genbeta{n}$ is monotone increasing and the last inequality is justified because $b_k-a_k\leq M$ by definition. If $a_k\geq b_k$, the above inequalities no longer work; however, we can choose our sequence $(N_k)_k$ without loss of generality so that it has the property
    \begin{equation}
        \gengamma{N_k-M}= \min\limits_{0\leq i\leq M} \gengamma{N_k-i},\label{eq_sequence_property}
    \end{equation}
    which ensures $a_k\leq b_k$.
    
    Begin with a generic increasing sequence $(P_k)_k$. Since Property 1 is not true, there exists a $k_0$ such that for all $n\geq P_{k_0}$, $\gengamma{n}\neq \min\limits_{0\leq i\leq M} \gengamma{n-i}$. For any $k\geq k_0$, let $0<l_k\leq M$ be such that $\gengamma{P_k-l_k}=\min\limits_{0\leq i\leq M}\gengamma{P_k-i}$. If $l_k=M$ we are done. Otherwise, since Property 1 is not true, $\gengamma{P_k-l_k}=\min\limits_{0\leq i\leq M}\gengamma{P_k+1-i}$. Repeating this argument, $\gengamma{P_k-l_k}=\min\limits_{0\leq i\leq M}\gengamma{N_k+M-l_k-i}$. Therefore the sequence $(P_k+M-l_k)_{k\geq k_0}$ has satisfies (\ref{eq_sequence_property}). Remove any duplicate terms from this sequence to get the desired strictly increasing sequence $(N_k)_k$.
    
    Now combining (\ref{eq_limsup-liminf_leq_bk-ak}) and (\ref{eq_bound_for_bk-ak}) we get, for each $k$,
    $$\limsup\limits_{n\rightarrow\infty}\gengamma{n}-\liminf\limits_{n\rightarrow\infty}\gengamma{n}\leq \gengamma{b_k}-\gengamma{a_k}\leq \frac{1}{b_k^\kep}\left(\frac{M}{\genbeta{a_k}^{1/(\ep-1)}}\right).$$
    Taking the limit as $k\rightarrow\infty$ shows that
    $$\limsup\limits_{n\rightarrow\infty}\gengamma{n}-\liminf\limits_{n\rightarrow\infty}\gengamma{n}\leq 0,$$
    or in other words, $$\limsup\limits_{n\rightarrow\infty}\gengamma{n}=\liminf\limits_{n\rightarrow\infty}\gengamma{n}=\lim\limits_{n\rightarrow\infty}\gengamma{n}.$$
    Substituting this into Lemma \ref{lem_limsup_liminf_relations}(a) and Lemma \ref{lem_limsup_liminf_relations}(b) shows
    \begin{equation*}\begin{split}
        \frac{1}{\lim\limits_{n\rightarrow\infty}\gengamma{n}^{\frac{1}{\ep-1}}} & \leq \genmu\lim\limits_{n\rightarrow\infty}\gengamma{n} \\
        \frac{1}{\lim\limits_{n\rightarrow\infty}\gengamma{n}^{\frac{1}{\ep-1}}}& \geq \genmu\lim\limits_{n\rightarrow\infty}\gengamma{n},
    \end{split}\end{equation*}
    meaning equality holds,
    $$\frac{1}{\lim\limits_{n\rightarrow\infty}\gengamma{n}^{\frac{1}{\ep-1}}}= \genmu\lim\limits_{n\rightarrow\infty}\gengamma{n}.$$
    Finally, solving this gives
    $$\lim\limits_{n\rightarrow\infty}\gengamma{n}=\result .$$
\end{proof}

Although we have a closed form expression for $\genbeta{n}$, in practice this can be computationally difficult value to find, a fact which is exacerbated with more possible order sizes and large inventory. Theorem \ref{thm_big_result} indicates that $(n/\genmu)^{\kep}$ can be a good approximation $\genbeta{n}$, since the two terms have the same asymptotic behavior in $n$.

So far we have established a closed form expression for $\genv{n}$ and have an understanding of the long-term behavior of $\genbeta{n}$, and therefore $\genv{n}$ as well. An important question is: ``How valuable is including the extra information of variable order sizes'' The next theorem shows that when two models have the same demand $\lambda(p,t)\genmu$, their asymptotic behavior in $n$ is the same. This can greatly simplify the modeling process, as a model with a complicated order size probability distribution can be approximated with a simpler one, thus improving practical computations. In particular this means that a model with variable order sizes can be approximated with a model of order size 1, the type of model widely used in the literature.

\begin{theorem}
    Comparable models (i.e. those with the same demand) have the same asymptotic behavior in $n$. \label{thm_same_long_term}
\end{theorem}

\begin{proof}
    Consider two models:
    \begin{equation*}\begin{array}{lr}
        \text{\textbf{Model 1}} & \\
        \text{Arrival Rate} & \lambda_1(p,t)=a_1(t)p^{-\ep} \\
        \text{Maximum order size} & M_1 \\
        \text{Order size distribution} & \qvec=(q_1,q_2,\ldots,q_{M_1}) \\
        \\
        \text{\textbf{Model 2}} & \\
        \text{Arrival Rate} & \lambda_2(p,t)=a_2(t)p^{-\ep} \\
        \text{Maximum order size} & M_2 \\
        \text{Order size distribution} & \textbf{w}=(w_1,w_2,\ldots,w_{M_2})
    \end{array}\end{equation*}
    such that their demands are equal, i.e. $$\lambda_1(p,t)\genmu=\lambda_2(p,t)\mu(\textbf{w}).$$ Substitute for $\lambda_1(p,t)$ and $\lambda_2(p,t)$ and solve $a_1(t)$ to get $a_1(t)=\frac{a_2(t)\mu(\textbf{w})}{\genmu}$. Define $A_i(t)=\int_t^T a_i(t)dt$ for $i=1,2$.
    
    Let $\delta>0$. By Theorems \ref{thm_optimal_expected_revenue} and \ref{thm_big_result}, there exists $N_1>0$ such that for all $n>N_1$,
    $$ \left|\genv{n} - \genmu \left(\frac{n}{\genmu}\right)^{\frac{\ep-1}{\ep}} A_1(t)^{\frac{1}{\ep}} \right| < \delta,$$
    and $N_2>0$ such that for all $n>N_2$,
    $$ \left|v_n(t; \textbf{w}) - \mu(\textbf{w}) \left(\frac{n}{\mu(\textbf{w})}\right)^{\frac{\ep-1}{\ep}} A_2(t)^{\frac{1}{\ep}} \right| < \delta,$$
    Note that 
    \begin{align*}
        \genmu \left(\frac{n}{\genmu}\right)^{\frac{\ep-1}{\ep}} A_1(t)^{\frac{1}{\ep}} & = \genmu \left(\frac{n}{\genmu}\right)^{\frac{\ep-1}{\ep}} \left(\int_t^T \frac{a_2(t)\mu(\textbf{w})}{\genmu}dt\right)^{\frac{1}{\ep}} \\
        & = \mu(\textbf{w})^{\frac{1}{\ep}} n^{\frac{\ep-1}{\ep}} \left(\int_t^T a_2(t)dt\right)^{\frac{1}{\ep}} \\
        & = \mu(\textbf{w}) \left(\frac{n}{\mu(\textbf{w})}\right)^{\frac{\ep-1}{\ep}} A_2(t)^{\frac{1}{\ep}}.
    \end{align*}
    Then for $n>\max(N_1,N_2)$,
    \begin{align*}
        \left| \genv{n} - v_n(t; \textbf{w})\right| & = \left| \genv{n} - \genmu \left(\frac{n}{\genmu}\right)^{\frac{\ep-1}{\ep}} A_1(t)^{\frac{1}{\ep}} + \mu(\textbf{w}) \left(\frac{n}{\mu(\textbf{w})}\right)^{\frac{\ep-1}{\ep}} A_2(t)^{\frac{1}{\ep}} - v_n(t; \textbf{w})\right| \\
        & \leq \left| \genv{n} - \genmu \left(\frac{n}{\genmu}\right)^{\frac{\ep-1}{\ep}} A_1(t)^{\frac{1}{\ep}}\right| + \left| \mu(\textbf{w}) \left(\frac{n}{\mu(\textbf{w})}\right)^{\frac{\ep-1}{\ep}} A_2(t)^{\frac{1}{\ep}} - v_n(t; \textbf{w})\right| \\
        & < 2\delta
    \end{align*}
    
    Thus proving the claim that the optimal expected revenue for two models with the same demand is asymptotically the same in $n$.
\end{proof}

\begin{corollary}
    A model with variable order sizes may be approximated with a model with only one order size.
\end{corollary}

\begin{proof}
     Consider two models. One with customer arrival rate $\lambda_1(p,t)$ and demand distribution $\qvec=(q_1,q_2,\ldots,q_{M_1})$ and another with customer arrival rate $\lambda_2(p,t)$ and a demand distribution of only one order size, i.e. the demand distribution is $\textbf{w}=(w_1,w_2,\ldots,w_{M_2})=(0,0,\ldots,0,1)$ such that their demand are equal, i.e. $\lambda_1(p,t)\genmu=\lambda_2(p,t)\mu(\textbf{w})$. By Theorem \ref{thm_same_long_term}, any model with the same demand has the same long-term behavior.
\end{proof}

The above proof gives a nice comparison to the basic model. It shows that generalizing the basic model to include information about how many items each customer buys will not significantly affect the optimal expected revenue for large $n$. That said, it is still useful to understand how close the optimal expected revenue is for comparable models for small values of $n$. In the next section we present some numerical results to address this question.

\subsection{Numerical Results}

So far we have established a closed form expression for the optimal expected revenue $\genv{n}$ and the optimal pricing strategy $\optp{n}$, whose formulas both involve $\genbeta{n}$. We have also found that comparable models have the same asymptotic behavior in $n$. We now want to answer the question of how close are the optimal expected revenues for two comparable models?

Since we will be comparing different models, our definitions must be expanded to include dependence on more variables. Recall from (\ref{eq_optimal_v}) that $$v_n(t; \textbf{q}) = \genmu\genbeta{n} A(t)^{1/\ep}=\genmu\genbeta{n}\left(\int_t^Ta(s)ds\right)^{1/\ep},$$
where $\lambda(p,t)=a(t)p^{-\ep}$. Define $v_n(t;\qvec,\lambda(p,t))$ as the optimal expected revenue for a model with order size distribution $\qvec$ and customer arrival rate $\lambda(p,t)=a(t)p^{-\ep}$.

Consider two comparable models, one with order size distribution $\qvec$ and arrival rate $\lambda_1(p,t)$, and the other with order  size distribution $\textbf{w}$ and arrival rate $\lambda_2(p,t)$ such that they are comparable, i.e. $\lambda_1(p,t)\genmu=\lambda_2(p,t)\mu(\textbf{w})$. Define the relative difference function between comparable models by
$$g_{n,t}(\qvec,\textbf{w}):= \frac{v_n(t; \qvec, \lambda_1(p,t))-v_n(t; \textbf{w}, \lambda_2(p,t))}{v_n(t; \textbf{w}, \lambda_2(p,t)))},$$
for $0\leq t\leq T$ and $n>0$. For simplicity of notation, we suppress the dependency of $g$ on $\lambda_1(p,t)$ and $\lambda_2(p,t)$.

To get numerical results, there is a large decision space for our variables. The ones which provide the most interesting comparisons are the demand elasticity $\ep$, the customer arrival rate $\lambda$ and the order size distribution $\qvec$. We will keep the other variables constant. For simplicity, assume that $a(t)=1$ and $T=1$.

Now we define several order size distributions in order to make comparisons.  Let $\delta_i$ be the distribution with full weight on the $i$-th component, i.e. each customer orders $i$ items. Let $\qvec_1=(0.25,0.25,0.25,0.25)$, i.e. there is an equal probability a person will buy 1, 2, 3, or 4 items, $\qvec_2=(0,0.4,0,0.6)$, and $\qvec_3=(0.7,0.1,0.2)$. Using these distributions, we calculated relative differences between the standard model and comparable variable order size models using these distributions. See tables.

\linespread{1}
\begin{table}[H]
\centering
\caption{Relative differences when $\ep = 1.25$, $\lambda\mu = 24$}
\sisetup{
table-number-alignment = center,
table-figures-integer = 3,
table-figures-decimal = 4
}
\begin{tabular}{
|S|
*{6}{S[table-auto-round]|}
}
\hline
{$n$}   & {$g_{n,0}(\single{2}, \single{1})$}  & {$g_{n,0}( \single{3}, \single{1})$}  & {$g_{n,0}( \single{4}, \single{1})$}  & {$g_{n,0}( \qvec_1, \single{1})$} & {$g_{n,0}(\qvec_2, \single{1})$} & {$g_{n,0}(\qvec_3, \single{1})$} \\
\hline
{1} & 	0.1486984 & 	0.2457309 & 	0.3195079 & 	0.2011244 & 	0.2619147 & 	0.0844718 \\ 
{2} & 	-0.1328733 & 	-0.0596255 & 	-0.0039329 & 	-0.0429339 & 	-0.0474088 & 	-0.0231511 \\ 
{3} & 	0.0046259 & 	-0.1775695 & 	-0.1288620 & 	-0.1046290 & 	-0.0876039 & 	-0.0742376 \\ 
{4} & 	-0.0775928 & 	0.0003246 & 	-0.2001560 & 	-0.1035342 & 	-0.1622745 & 	-0.0465427 \\ 
{5} & 	-0.0086428 & 	-0.0597425 & 	-0.0040568 & 	-0.0556868 & 	-0.0334583 & 	-0.0394805 \\ 
{6} & 	-0.0553971 & 	-0.1040868 & 	-0.0510274 & 	-0.0643441 & 	-0.0790423 & 	-0.0388159 \\ 
{7} & 	-0.0110279 & 	-0.0152622 & 	-0.0877650 & 	-0.0632933 & 	-0.0566836 & 	-0.0342953 \\ 
{8} & 	-0.0434383 & 	-0.0475339 & 	-0.1176606 & 	-0.0567364 & 	-0.0875978 & 	-0.0308723 \\ 
{9} & 	-0.0112095 & 	-0.0745633 & 	-0.0197554 & 	-0.0493902 & 	-0.0395037 & 	-0.0286288 \\ 
{10} & 	-0.0359328 & 	-0.0172763 & 	-0.0442650 & 	-0.0481420 & 	-0.0635196 & 	-0.0265025 \\ 
{50} & 	-0.0091048 & 	-0.0131046 & 	-0.0165278 & 	-0.0142713 & 	-0.0190570 & 	-0.0074057 \\ 
{100} & 	-0.0050270 & 	-0.0054788 & 	-0.0140861 & 	-0.0081375 & 	-0.0107457 & 	-0.0041680 \\ 
{200} & 	-0.0027641 & 	-0.0043028 & 	-0.0077772 & 	-0.0045826 & 	-0.0060041 & 	-0.0023240 \\ 
{500} & 	-0.0012438 & 	-0.0020015 & 	-0.0035185 & 	-0.0021120 & 	-0.0027470 & 	-0.0010601 \\ 
\hline
\end{tabular}

~\\
~\\

\caption{Relative differences when $\ep = 1.6$, $\lambda\mu = 24$}
\sisetup{
table-number-alignment = center,
table-figures-integer = 3,
table-figures-decimal = 4
}
\begin{tabular}{
|S|
*{6}{S[table-auto-round]|}
}
\hline
{$n$}   & {$g_{n,0}(\single{2}, \single{1})$}  & {$g_{n,0}( \single{3}, \single{1})$}  & {$g_{n,0}( \single{4}, \single{1})$}  & {$g_{n,0}( \qvec_1, \single{1})$} & {$g_{n,0}(\qvec_2, \single{1})$} & {$g_{n,0}(\qvec_3, \single{1})$} \\
\hline
{1} & 	0.2968396 & 	0.5098036 & 	0.6817928 & 	0.4100272 & 	0.5467896 & 	0.1642178 \\ 	
{2} & 	-0.1386604 & 	0.0027869 & 	0.1170193 & 	0.0314461 & 	0.0273524 & 	0.0248281 \\ 	
{3} & 	0.0492839 & 	-0.1886366 & 	-0.0962102 & 	-0.0666030 & 	-0.0275438 & 	-0.0468310 \\ 	
{4} & 	-0.0883758 & 	0.0613290 & 	-0.2147820 & 	-0.0767443 & 	-0.1551242 & 	-0.0318600 \\ 	
{5} & 	0.0148513 & 	-0.0440319 & 	0.0648671 & 	-0.0347320 & 	0.0081174 & 	-0.0288536 \\ 	
{6} & 	-0.0659089 & 	-0.1201062 & 	-0.0198732 & 	-0.0510670 & 	-0.0721069 & 	-0.0309355 \\ 	
{7} & 	0.0038844 & 	0.0154084 & 	-0.0848347 & 	-0.0537945 & 	-0.0356941 & 	-0.0288722 \\ 	
{8} & 	-0.0530573 & 	-0.0421870 & 	-0.1367442 & 	-0.0494446 & 	-0.0903910 & 	-0.0269208 \\ 	
{9} & 	-0.0007675 & 	-0.0896321 & 	0.0140724 & 	-0.0440299 & 	-0.0232646 & 	-0.0256774 \\ 	
{10} & 	-0.0446747 & 	0.0017949 & 	-0.0304869 & 	-0.0444376 & 	-0.0661833 & 	-0.0243602 \\ 	
{50} & 	-0.0122147 & 	-0.0153336 & 	-0.0176229 & 	-0.0163637 & 	-0.0233338 & 	-0.0083598 \\ 	
{100} & 	-0.0068493 & 	-0.0053921 & 	-0.0189806 & 	-0.0097506 & 	-0.0135933 & 	-0.0049088 \\ 	
{200} & 	-0.0038069 & 	-0.0054071 & 	-0.0106336 & 	-0.0056619 & 	-0.0077687 & 	-0.0028207 \\ 	
{500} & 	-0.0017300 & 	-0.0025837 & 	-0.0048704 & 	-0.0026842 & 	-0.0036283 & 	-0.0013240 \\ 	

\hline
\end{tabular}

\end{table}

\begin{table}[H]
\centering
\caption{Relative differences when $\ep = 1.6$, $\lambda\mu = 12$}
\sisetup{
table-number-alignment = center,
table-figures-integer = 3,
table-figures-decimal = 4
}
\begin{tabular}{
|S|
*{6}{S[table-auto-round]|}
}
\hline
{$n$}   & {$g_{n,0}(\single{2}, \single{1})$}  & {$g_{n,0}( \single{3}, \single{1})$}  & {$g_{n,0}( \single{4}, \single{1})$}  & {$g_{n,0}( \qvec_1, \single{1})$} & {$g_{n,0}(\qvec_2, \single{1})$} & {$g_{n,0}(\qvec_3, \single{1})$} \\
\hline
{1} & 	0.2968396 & 	0.5098036 & 	0.6817928 & 	0.4100272 & 	0.5467896 & 	0.1642178 \\ 	
{2} & 	-0.1386604 & 	0.0027869 & 	0.1170193 & 	0.0314461 & 	0.0273524 & 	0.0248281 \\ 	
{3} & 	0.0492839 & 	-0.1886366 & 	-0.0962102 & 	-0.0666030 & 	-0.0275438 & 	-0.0468310 \\ 	
{4} & 	-0.0883758 & 	0.0613290 & 	-0.2147820 & 	-0.0767443 & 	-0.1551242 & 	-0.0318600 \\ 	
{5} & 	0.0148513 & 	-0.0440319 & 	0.0648671 & 	-0.0347320 & 	0.0081174 & 	-0.0288536 \\ 	
{6} & 	-0.0659089 & 	-0.1201062 & 	-0.0198732 & 	-0.0510670 & 	-0.0721069 & 	-0.0309355 \\ 	
{7} & 	0.0038844 & 	0.0154084 & 	-0.0848347 & 	-0.0537945 & 	-0.0356941 & 	-0.0288722 \\ 	
{8} & 	-0.0530573 & 	-0.0421870 & 	-0.1367442 & 	-0.0494446 & 	-0.0903910 & 	-0.0269208 \\ 	
{9} & 	-0.0007675 & 	-0.0896321 & 	0.0140724 & 	-0.0440299 & 	-0.0232646 & 	-0.0256774 \\ 	
{10} & 	-0.0446747 & 	0.0017949 & 	-0.0304869 & 	-0.0444376 & 	-0.0661833 & 	-0.0243602 \\ 	
{50} & 	-0.0122147 & 	-0.0153336 & 	-0.0176229 & 	-0.0163637 & 	-0.0233338 & 	-0.0083598 \\ 	
{100} & 	-0.0068493 & 	-0.0053921 & 	-0.0189806 & 	-0.0097506 & 	-0.0135933 & 	-0.0049088 \\ 	
{200} & 	-0.0038069 & 	-0.0054071 & 	-0.0106336 & 	-0.0056619 & 	-0.0077687 & 	-0.0028207 \\ 	
{500} & 	-0.0017300 & 	-0.0025837 & 	-0.0048704 & 	-0.0026842 & 	-0.0036283 & 	-0.0013240 \\ 	
	
\hline
\end{tabular}
\end{table}

\linespread{1.5}

The data shows some interesting properties about variable order size models and their comparable standard models. First, when $n$ is small, the overselling effect is pronounced, making the optimal expected revenue larger for variable order sizes. However, the overselling effect appears to dissipate rather quickly. Once it does, we see that models with variable order sizes have a lower optimal expected revenue compared to the standard models. This makes sense because of how the model relates to real world interactions. In the standard model, if a person wishes to buy multiple items, they must do so in a small interval of time, paying slightly more for each successive purchase because prices are updated in real time. But when a customer buys multiple items at once, they pay the same price for each, resulting in slightly less revenue overall. 

The most interesting observation is upon comparing the data tables for $\ep=1.6$, $\lambda\mu=24$ and for $\ep=1.6,\lambda\mu=12$. They are identical, despite having different demands. Since the average order size $\mu(\qvec)$ for a particular model is determined by its order size distribution $\qvec$, the difference in demand is caused by a difference in the customer arrival rate $\lambda$. These tables indicate that $\lambda$ does not impact the relative difference between a variable order size model and its comparable standard model. This observation is true in general and is proved in the next theorem.

\begin{theorem}
    For any particular $0\leq t<T$ and $n>0$, the function $g_{n,t}$, the relative difference function between the optimal expected revenues for comparable models, is completely determined by their order size distributions. \label{thm_comparable_model_differences}
\end{theorem}

\begin{proof}
    Consider a model $\mathscr{M}_1$ with order size distribution $\qvec$ and customer arrival rate $\lambda_1(p,t)$ and another model $\mathscr{M}_2$ with order size distribution $\textbf{w}$ and arrival rate $\lambda_2(p,t)$ such that the demands of $\mathscr{M}_1$ and $\mathscr{M}_2$ are equal, i.e. $\lambda_1(p,t)\genmu=\lambda_2(p,t)\mu(\textbf{w})$. Recall the optimal expected revenue from Theorem \ref{thm_optimal_expected_revenue},
    $$ v_n(t; \textbf{q}) = \genmu\genbeta{n} A(t)^{1/\ep},$$
    where $A(t)=\int_t^T a(t)dt$ for $\lambda(p,t)=a(t)p^{-\ep}$. To accommodate the different models we add the customer arrival rate dependency, so define $A(t; \lambda(p,t))=\int_t^Ta(t)dt$. Note that for a constant $c$, $A(t; c\lambda(p,t))=\int_t^Tca(t)dt=c\int_t^Ta(t)dt=cA(t; \lambda(p,t))$. Then the relative difference is

    \begin{align*}
        g_{n,t}(\qvec, \textbf{w}) & = \frac{v_n(t; \qvec, \lambda_1(p,t))-v_n(t; \textbf{w}, \lambda_2(p,t))}{v_n(t; \textbf{w}, \lambda_2(p,t))} \\
        & = \frac{v_n(t; \qvec, \lambda_2(p,t)\mu(\textbf{w})\genmu^{-1})-v_n(t; \textbf{w}, \lambda_2(p,t))}{v_n(t; \textbf{w}, \lambda_2(p,t))} \\
        & = \frac{\genmu\genbeta{n} A(t; \lambda_2(p,t)\mu(\textbf{w})\genmu^{-1})^{1/\ep}-\mu(\textbf{w})\beta_n(\textbf{w}) A(t; \lambda_2(p,t))^{1/\ep}}{\mu(\textbf{w})\beta_n(\textbf{w}) A(t; \lambda_2(p,t))^{1/\ep}} \\
        & = \frac{\genmu\genbeta{n}(\mu(\textbf{w})\genmu^{-1})^{1/\ep}A(t; \lambda_2(p,t))^{1/\ep}-\mu(\textbf{w})\beta_n(\textbf{w}) A(t; \lambda_2(p,t))^{1/\ep}}{\mu(\textbf{w})\beta_n(\textbf{w}) A(t; \lambda_2(p,t))^{1/\ep}} \\
        & = \frac{\genmu^{(\ep-1)/\ep}\mu(\textbf{w})^{1/\ep}\genbeta{n} - \beta_n(\textbf{w})}{ \beta_n(\textbf{w})}.
    \end{align*}
    Therefore the relative difference for a particular $n$ and $t$ between comparable models $\mathscr{M}_1$ and $\mathscr{M}_2$ is completely determined by their order size distributions $\qvec$ and $\textbf{w}$.
\end{proof}

\textbf{Remark:} Since $g_{n,t}$ is determined by order size distributions, this means that $g_{n,t}$ does not depend on the customer arrival rates $\lambda_1(p,t)$ and $\lambda_2(p,t)$ and justifies the suppression of the arrival rates in the definition of $g_{n,t}$. Theorem \ref{thm_comparable_model_differences} also proves the surprising fact that two models with same order size distribution but different sales rates still have the same relative difference to their comparable models with order size distribution $\delta_1$.

\section{Conclusion}

We have considered a dynamic pricing model for constant demand elasticity. Our generalization of the standard model was to allow customers to buy multiple items at a time. We established closed form expressions for the optimal expected revenue and optimal pricing strategies. Then we showed that two comparable models have the same asymptotic behavior in their inventory size $n$. We also showed that the relative difference between two comparable models is determined by their order size distributions, and that their customer arrival rates did not have an effect. Lastly, we provided a few numerical computations to show applications of our results.

An important insight our work provides is that a model with variable order sizes may be approximated by a model where customers can only order one item at a time. However, in practice customers usually have the option to buy more than one item at a time, which makes our variable order size model closer to reality.

This paper has only considered one specific model, when demand elasticity is constant. While these ideas can be applied to models with different customer arrival rate functions, we have shown here that the generalization to variable order sizes is non-trivial. In future work we plan to examine other types of demand functions to see if we can find similar results in these cases. Another avenue we plan to explore is variable order sizes combined with other model generalizations. For example, Helmes and Schlosser (2013) \cite{Helmes} work with constant demand elasticity but introduce advertising effects.

We believe that variable order sizes provides an interesting generalization to dynamic pricing problems, both from a theoretical and practical standpoint. It reveals more about the general structure of the problem by highlighting the role the average order size plays. We have also shown there are compelling comparisons between models which have the same demand but different order size probabilities.

\subsection*{Acknowledgement}

This research was supported in part by the Simons Foundation (grant award number 246271). We would also like to thank Kurt Helmes for his helpful insights on the asymptotic behavior of $\genbeta{n}$.

\section{Appendix}

\subsection{Model Details}

Our model looks like a compound Poisson process, but is actually a random time change of a compound Poisson process due to the dependence of demand $\lambda(p,t)$ on price $p$ and time $t$. Let $X_i$ for $1\leq i$ be IID variables with distribution $\qvec$, i.e. $P(X_i=n)=q_n$ for all $1\leq n\leq M$). Without loss of generality, assume the starting inventory equals $M$. (As we could always extend a probability distribution up to max order size $M$ is necessary). Let $N(t)$ be a counting process which counts the arrival of customers which has intensity $\lambda(p,t)$. We now have the pieces to define our process. Let $Y(t)$ be the inventory level at time $t$, which we may write $$Y(t)=M-\sum_{i=1}^N(t)X_i$$.

The process is a random time change of a compound process where the time change is defined by $\Lambda(t)=\int_0^t\lambda(p(s),s)ds$. Note that here price is dependent on time, because the time will affect the remaining inventory, which will in turn affect price. See [] for more details.

\subsection{Technical Lemmas}

The rest of the appendix provides proofs of technical lemmas. The first is used to prove another one of the lemmas, but the other two are used in the proof of Theorem \ref{thm_big_result} itself.

\begin{restatable}{lemma}{lemflimit}
    The function
    $$f(n; \textbf{q}):=n\left(1-\sum_{i=1}^Mq_i\left(\frac{n-i}{n}\right)^{\kep}\right)$$
    is decreasing for $n> M$, $n\in\R$ and
    $$\lim\limits_{n\rightarrow\infty}f(n; \textbf{q})=\mu(\textbf{q})\left(\kep\right).$$ 
    \label{lem_f_limit}
\end{restatable}

\begin{proof}
    Let $n> M$, $n\in\R$. First we show that $f$ has a limiting value. Using L'Hôpital's rule,
    \begin{equation*}\begin{split}
        \lim\limits_{n\rightarrow\infty}f(n; \textbf{q}) = &\lim\limits_{n\rightarrow\infty}\frac{1-\sum_{i=1}^Mq_i\left(\frac{n-i}{n}\right)^{\kep}}{1/n}\\
        = & \lim\limits_{n\rightarrow\infty}\frac{-\sum_{i=1}^Mq_i\left(\kep\right)\left(\frac{n-i}{n}\right)^{-1/\ep}\left(\frac{i}{n^2}\right)}{-1/n^2} \\
        = & \lim\limits_{n\rightarrow\infty}\sum_{i=1}^Miq_i\left(\kep\right)\left(\frac{n-i}{n}\right)^{-1/\ep} \\
        = & \sum_{i=1}^Miq_i\left(\kep\right) \\
        = & \mu(\textbf{q})\left(\kep\right)
    \end{split}\end{equation*}
    To get $f$ decreasing we look at $f''$ to show that $f$ is concave up, and since $f$ has a limiting value, $f$ will be decreasing. First calculate $f'$ by
    \begin{equation*}\begin{split}
        f'(n; \textbf{q}) = & n\left(-\sum_{i=1}^Mq_i\left(\kep\right)\left(\frac{n-i}{n}\right)^{-1/\ep}\left(\frac{i}{n^2}\right)\right) +\left(1-\sum_{i=1}^Mq_i\left(\frac{n-i}{n}\right)^{\kep}\right) \\
        = & 1 - \sum_{i=1}^M q_i\left(\left(\kep\right)\left(\frac{n-i}{n}\right)^{-1/\ep}\left(\frac{i}{n}\right)+\left(\frac{n-i}{n}\right)^{\kep}\right) \\
        = & 1 - \sum_{i=1}^M q_i\left(\frac{n-i}{n}\right)^{-1/\ep}\left(\frac{i\ep-i}{n\ep}+\frac{n-i}{n}\right)\\
        = & 1 - \sum_{i=1}^M q_i\left(\frac{n}{n-i}\right)^{1/\ep}\left(\frac{n\ep-i}{n\ep}\right).
    \end{split}\end{equation*}
    For the second derivative we have
    \begin{align*}
        f''(n; \textbf{q}) = & - \sum_{i=1}^M q_i\left[\left(\frac{n}{n-i}\right)^{1/\ep}\left(\frac{i\ep}{n^2\ep^2}\right)+\left(\frac{1}{\ep}\right)\left(\frac{n}{n-i}\right)^{1/\ep-1}\left(\frac{-i}{(n-i)^2}\right)\left(\frac{n\ep-i}{n\ep}\right)\right] \\
        = & - \sum_{i=1}^M \frac{iq_i}{\ep}\left(\frac{n}{n-i}\right)^{1/\ep}\left[\frac{1}{n^2}-\left(\frac{n}{n-i}\right)^{-1}\left(\frac{1}{(n-i)^2}\right)\left(\frac{n\ep-i}{n\ep}\right)\right] \\
        = & - \frac{iq_i}{\ep n^2}\left(\frac{n}{n-i}\right)^{1/\ep}\left[1-\frac{n\ep-i}{n \ep -i \ep}\right]\\
        > & 0.
    \end{align*}
    Note that the inequality is justified since $\ep>1$ and thus the bracketed expression is negative. Thus $f$ is concave up, completing the proof of Lemma \ref{lem_f_limit}.
\end{proof}

\lemgammabounds*

\begin{proof}
   \textbf{Proof of part \ref{lem_gamma_bounds_a}:}  By definition, $\genbeta{n}\geq 0$, and therefore $\gengamma{n}\geq 0$ and $\liminf\limits_{n\rightarrow\infty}\gengamma{n}\geq 0$. Assume to the contrary that $\liminf\limits_{n\rightarrow\infty}\gengamma{n}=0$. Construct a subsequence $(\gengamma{N_k})_k$ of $\gengamma{n}$ such that $N_k>M$, $\lim\limits_{k\rightarrow\infty}\gengamma{N_k}=0$, and $\gengamma{N_k}\leq \min\limits_{0< i\leq M} \gengamma{N_k-i}$. Let $\delta>0$. Then there exists $k>0$ such that $\gengamma{N_k}<\delta$. By equation (\ref{eq_gengamma}),
    \begin{align}
        \kep & = \gengamma{N_k}^{\frac{1}{\ep-1}}N_k\left(\gengamma{N_k}-\sum_{i=1}^Mq_i\gengamma{N_k-i}\left(\frac{N_k-i}{N_k}\right)^{\kep}\right) \nonumber \\
        & \leq \gengamma{N_k}^{\frac{1}{\ep-1}}N_k\left(\gengamma{N_k}-\sum_{i=1}^Mq_i\gengamma{N_k}\left(\frac{N_k-i}{N_k}\right)^{\kep}\right) \nonumber \\
        & = \gengamma{N_k}^{\frac{\ep}{\ep-1}}N_k\left(1-\sum_{i=1}^Mq_i\left(\frac{N_k-i}{N_k}\right)^{\kep}\right) \nonumber \\
        & =: \gengamma{N_k}^{\frac{\ep}{\ep-1}}f(N_k; \qvec) \label{eq_gengamma_and_f}
    \end{align}
    By Lemma \ref{lem_f_limit}, $f(n;\qvec)$ is decreasing for $n> M$ and therefore $f(N_k;\qvec)\leq f(M; \qvec)$. Substituting into the above, we get
     $$ \kep\leq \gengamma{N_k}^{\frac{\ep}{\ep-1}}f(M; \qvec)\leq \delta^{\frac{\ep}{\ep-1}}f(M; \qvec),$$
    but this is a contradiction since $f(M; \qvec)$ is a constant and $\delta$ is arbitrary. Therefore,
    $$\liminf\limits_{n\rightarrow\infty}\gengamma{n}\neq 0,$$
    completing the proof of part \ref{lem_gamma_bounds_a}.
    
    \textbf{Proof of part \ref{lem_gamma_bounds_b}:} Suppose there exists an increasing sequence $(N_k)_k$ such that $N_1>M$ and $\gengamma{N_k}=\min\limits_{0\leq i\leq M} \gengamma{N_k-i}$ for all $k$. Without loss of generality, assume that $(N_k)_k$ contains every $n>M$ such that $\gengamma{n}=\min\limits_{0\leq i\leq M} \gengamma{n-i}$. First we get a bound for $\liminf\limits_{k\rightarrow\infty}\gengamma{N_k}$, then we show $\liminf\limits_{k\rightarrow\infty}\gengamma{N_k}=\liminf\limits_{n\rightarrow\infty}\gengamma{n}$.
    
    By assumption, (\ref{eq_gengamma_and_f}) holds for each $N_k$,
    $$\kep\leq\gengamma{N_k}^{\frac{\ep}{\ep-1}}f(N_k; \qvec),$$
    or
    $$\gengamma{N_k} \geq \left(\frac{\ep-1}{\ep f(N_k;\qvec)}\right)^{\kep}.$$
    Let $\delta>0$. By Lemma \ref{lem_f_limit}, there exists a $K\geq M$ such that for all $k\geq K$,
    $f(N_k; \qvec)\leq \genmu\left(\frac{\ep-1}{\ep}\right)(1+\delta)$\, and thus
    $$\gengamma{N_k} \geq \left(\frac{1}{\genmu(1+\delta)}\right)^{\kep}.$$
    Therefore
    $$\liminf_{k\rightarrow\infty}\gengamma{N_k} \geq \liminf_{k\rightarrow\infty} \left(\frac{1}{\genmu(1+\delta)}\right)^{\kep}=\left(\frac{1}{\genmu(1+\delta)}\right)^{\kep},$$
    but $\delta$ was arbitrary and so
    $$\liminf\limits_{k\rightarrow\infty}\gengamma{N_k}\geq \left(\frac{1}{\genmu}\right)^{\kep}=\result.$$
    
    For each $k\geq N_1$, define $a_k$ such that $N_{a_k}\leq k < N_{a_k+1}$. Since $(N_k)_k$ contains every $n>M$ such that $\gengamma{n}=\min\limits_{0\leq i\leq M} \gengamma{n-i}$, then $N_{a_k+1}$ is the first occurrence of any $n>N_{a_k}$ such that $\gengamma{n}\leq\gengamma{N_{a_k}}$. Thus $$\gengamma{k}> \gengamma{N_{a_k}}.$$ Hence
    $$\liminf\limits_{k\rightarrow\infty}\gengamma{k}\geq \liminf\limits_{k\rightarrow\infty}\gengamma{N_{a_k}}\geq \liminf\limits_{k\rightarrow\infty}\gengamma{N_k}\geq \result,$$
    completing the proof of part (b).
    
    \textbf{Proof of part \ref{lem_gamma_bounds_c}} Assume there exists an $N\geq M$ such that $\gengamma{N}\geq\max\limits_{0<i\leq M}\gengamma{N-i}$. Then equation (\ref{eq_gengamma_and_f}) holds, except with the inequality reversed, that is
    \begin{align}
        \kep & \geq \gengamma{N}^{\frac{\ep}{\ep-1}}f(N; \qvec) \nonumber
    \end{align}
    By lemma \ref{lem_f_limit}, $f$ is decreasing and $\lim\limits_{n\rightarrow\infty}f(n; \textbf{q})=\mu(\textbf{q})\left(\kep\right).$ Therefore, substituting into the above, we get
     $$ \kep\geq \gengamma{N}^{\frac{\ep}{\ep-1}}\mu(\textbf{q})\left(\kep\right),$$
    or
    $$\genmu^{\negkep} \geq \gengamma{N}.$$
    Suppose $K$ is the smallest $K>N$ such that $\gengamma{K}\geq\gengamma{N}$. Then $\gengamma{K}=\max\limits_{0\leq i\leq M}\gengamma{K-i}$. By what we just showed, $\gengamma{K}\leq \result$. By repetition of this argument, $\gengamma{n}\leq \result$ for all $n\geq N$ and so
    $$\limsup\limits_{n\rightarrow\infty}\gengamma{n}\leq\result.$$
\end{proof}

\lemrelations*

\begin{proof}
    (Of part 1) Starting with the definition equation for $\genbeta{n}$ in equation (\ref{eq_genbeta}),
    \begin{equation*}\begin{split}
        \left(\kep\right)\frac{1}{\genbeta{n}^{\frac{1}{\ep-1}}}& = \beta_n(\textbf{q})-\sum_{i=1}^Mq_i \beta_{n-i}(\textbf{q}) \\
         & = \sum_{i=1}^Mq_i (\genbeta{n}-\genbeta{n-i}) \\
         & = \sum_{i=1}^Mq_i \left[(\genbeta{n}-\genbeta{n-1})+\sum_{j=2}^{i}(\genbeta{n-j+1}-\genbeta{n-j})\right]\\
         & = \sum_{i=1}^Mq_i (\genbeta{n}-\genbeta{n-1})+\sum_{i=2}^Mq_i\sum_{j=2}^{i}(\genbeta{n-j+1}-\genbeta{n-j} )\\
         & = \sum_{i=1}^M iq_i (\genbeta{n}-\genbeta{n-1})-\sum_{i=2}^M (i-1)q_i (\genbeta{n}-\genbeta{n-1}) \\
         & \qquad + \sum_{i=2}^M\sum_{j=2}^{i}q_i(\genbeta{n-j+1}-\genbeta{n-j} )\\
         & = \genmu \Delta\genbeta{n}- \sum_{i=2}^M\sum_{j=2}^{i}q_i(\Delta\genbeta{n}-\Delta\genbeta{n-j+1}),
    \end{split}  \end{equation*}
    where $\Delta\genbeta{k}=\genbeta{k}-\genbeta{k-1}$ for all $k$. Substitute $\genbeta{n}=n^\kep\gengamma{n}$ into the left hand side of the above to get
    \begin{equation}
        \left(\kep\right)\frac{1}{n^{\frac{1}{\ep}}\gengamma{n}^\frac{1}{\ep-1}} = \genmu \Delta\genbeta{n}- \sum_{i=2}^M\sum_{j=2}^{i}q_i(\Delta\genbeta{n}-\Delta\genbeta{n-j+1}). \label{eq_deltas}
    \end{equation}
    Our strategy will be to examine individual terms of this equation. Through this analysis, we will ultimately arrive at the inequality for this part of the lemma:
    \begin{equation*}
        \frac{1}{\limsup\limits_{n\rightarrow\infty}\gengamma{n}^{\frac{1}{\ep-1}}}\leq \genmu\liminf\limits_{n\rightarrow\infty}\gengamma{n}.
    \end{equation*}
    
    Choose $\delta>0$ small (eventually we allow $\delta\rightarrow 0$), and $N_\delta$ large enough such that \begin{equation}
        \frac{1}{\gengamma{n}^\frac{1}{\ep-1}} \geq \frac{1}{\limsup\limits_{n\rightarrow\infty}\gengamma{n}^{\frac{1}{\ep-1}}} - \delta \label{eq_bound_delta}
    \end{equation}
    for all $n\geq N_\delta$. Such $N_\delta$ exists because if $\gengamma{n}\leq\lsup{n}\gengamma{n}$, then  (\ref{eq_bound_delta}) is satisfied trivially, and if $\gengamma{n}>\lsup{n}\gengamma{n}$, then we can choose $N_\delta$ large enough to make $\gengamma{n}$ close enough to $\lsup{n}\gengamma{n}$ to satisfy (\ref{eq_bound_delta})
    
    Let $N>N_\delta$. Eventually we will take a limit as $N\rightarrow\infty$, but for now we add up equation (\ref{eq_deltas}) from $n=N_\delta$ to $N$ to obtain
    \begin{equation*}
        \sum\limits_{n=N_\delta}^N \left(\kep\right)\frac{1}{n^{\frac{1}{\ep}}\gengamma{n}^\frac{1}{\ep-1}} = \sum\limits_{n=N_\delta}^N\genmu\Delta\genbeta{n}-\sum\limits_{n=N_\delta}^N \sum_{i=2}^M\sum_{j=2}^{i}q_i(\Delta\genbeta{n}-\Delta\genbeta{n-j+1}).
    \end{equation*}
    Note that $\sum\limits_{n=N_\delta}^N\Delta\genbeta{n}=\genbeta{N}-\genbeta{N_\delta-1}$, making the above
    \begin{equation}
        \sum\limits_{n=N_\delta}^N \left(\kep\right)\frac{1}{n^{\frac{1}{\ep}}\gengamma{n}^\frac{1}{\ep-1}} =\genmu(\genbeta{N}-\genbeta{N_\delta-1}) - S, \label{eq_deltas_sum}
    \end{equation}
    where
    $$S=\sum_{i=2}^M\sum_{j=2}^{i}q_i\left[(\genbeta{N}-\genbeta{N_\delta-1})-(\genbeta{N-j+1}-\genbeta{N_\delta-j})\right].$$
    We have that $$\sum\limits_{n=N_\delta}^N\frac{1}{n^{\frac{1}{\ep}}}\geq \int_{N_\delta}^{N+1} x^{-1/\ep}dx \geq \int_{N_\delta}^{N} x^{-1/\ep}dx=\left(\frac{\ep}{\ep-1}\right)\left(N^\kep-N_\delta^\kep\right).$$
    Applying this to (\ref{eq_deltas_sum}) and using (\ref{eq_bound_delta}) we get
    \begin{equation*}
       \left(N^\kep-N_\delta^\kep\right) \left(\frac{1}{\limsup\limits_{n\rightarrow\infty}\gengamma{n}^{\frac{1}{\ep-1}}} - \delta\right) \leq \genmu(\genbeta{N}-\genbeta{N_\delta-1}) - S.
    \end{equation*}
    so dividing by $N^{\frac{\ep-1}{\ep}}$,
    \begin{equation}
       \left(1-\left(\frac{N_\delta}{N}\right)^\kep\right) \left(\frac{1}{\limsup\limits_{n\rightarrow\infty}\gengamma{n}^{\frac{1}{\ep-1}}} - \delta\right) \leq \genmu\left(\frac{\genbeta{N}}{N^\kep}-\frac{\genbeta{N_\delta-1}}{N^\kep}\right) - \frac{S}{N^\kep}. \label{eq_ineq_to_show_liminf}
    \end{equation}
    Now examine $S$,
    \begin{align}
        S & = \sum_{i=2}^M\sum_{j=2}^{i}q_i\left[(\genbeta{N}-\genbeta{N_\delta-1})-(\genbeta{N-j+1}-\genbeta{N_\delta-j})\right] \nonumber \\
        & = \sum_{i=2}^M\sum_{j=2}^{i}q_i\left[(\genbeta{N}-\genbeta{N-j+1}))-(\genbeta{N_\delta-1}-\genbeta{N_\delta-j})\right] \nonumber \\
        & = \sum_{i=2}^M\sum_{j=2}^{i}q_i\left[\sum\limits_{k=0}^{j}\Delta\genbeta{N-k}-(\genbeta{N_\delta-1}-\genbeta{N_\delta-j})\right]. \label{eq_S}
    \end{align}
    Recall that $\genbeta{n}$ is increasing in $n$, thus for any $n\geq M$,
    \begin{align}
        \Delta\genbeta{n} & =\genbeta{n}-\genbeta{n-1} \nonumber \\
        & = \genbeta{n}-\sum\limits_{i=1}^M q_i\genbeta{n-1} \nonumber\\
        & \leq \genbeta{n}-\sum\limits_{i=1}^M q_i\genbeta{n-i} \nonumber\\
        & =\left(\kep\right)\left(\frac{1}{\genbeta{n}^{1/(\ep-1)}}\right), \label{eq_beta_step_difference}
    \end{align}
    where the last equality follows from the equation for $\genbeta{n}$. Using this with equation (\ref{eq_S}), we get
    \begin{align*}
        \left|\frac{S}{N^\kep}\right| & \leq \sum_{i=2}^M\sum_{j=2}^{i}q_i\left[\sum\limits_{k=0}^{j}\left|\frac{\Delta\genbeta{N-k}}{N^\kep}\right|+\left|\frac{\genbeta{N_\delta-1}-\genbeta{N_\delta-j}}{N^\kep}\right|\right] \\
        & \leq \sum_{i=2}^M\sum_{j=2}^{i}q_i\left[\sum\limits_{k=0}^{j}\left(\kep\right)\left|\frac{1}{\genbeta{N-k}^{1/(\ep-1)}N^\kep}\right|+\left|\frac{\genbeta{N_\delta-1}-\genbeta{N_\delta-j}}{N^\kep}\right|\right],
    \end{align*}
    which limits to 0 as $N\rightarrow\infty$. Therefore
    $$\liminf\limits_{N\rightarrow\infty}\frac{S}{N^\kep}=\lim\limits_{N\rightarrow\infty}\frac{S}{N^\kep}=0.$$
    Now take $\liminf$ as $N\rightarrow\infty$ of inequality (\ref{eq_ineq_to_show_liminf}) to get
    \begin{align}
        \frac{1}{\limsup\limits_{n\rightarrow\infty}\gengamma{n}^{\frac{1}{\ep-1}}} - \delta & \leq \linf{N}\left[\genmu\left(\frac{\genbeta{N}}{N^\kep}-\frac{\genbeta{N_\delta-1}}{N^\kep}\right) - \frac{S}{N^\kep}\right] \nonumber \\
        & = \genmu\linf{N}\frac{\genbeta{N}}{N^\kep}-\genmu\lim\limits_{N\rightarrow\infty}\frac{\genbeta{N_\delta-1}}{N^\kep} - \lim\limits_{N\rightarrow\infty}\frac{S}{N^\kep} \nonumber \\
        & = \genmu\linf{N}\frac{\genbeta{N}}{N^\kep},
        \label{eq_liminf_ineq_with_s}
    \end{align}
    where splitting of the $\liminf$ is justified because the second and third terms limit to 0. Since (\ref{eq_liminf_ineq_with_s}) is true for all $\delta>0$, we have the relation
    \begin{equation}
        \frac{1}{\limsup\limits_{n\rightarrow\infty}\gengamma{n}^{\frac{1}{\ep-1}}}\leq \genmu\liminf\limits_{n\rightarrow\infty}\gengamma{n}. \label{eq_limsup_denom}
    \end{equation}
    
    (Of part 2) The proof of part 2 is developed in a nearly identical way to part 1, but in reverse. Note however that the analog of Equation (\ref{eq_bound_delta}) is
   \begin{equation*}
        \frac{1}{\gengamma{n}^\frac{1}{\ep-1}} \leq \frac{1}{\liminf\limits_{n\rightarrow\infty}\gengamma{n}^{\frac{1}{\ep-1}}} + \delta, \label{eq_bound_xi}
    \end{equation*}
    which requires the extra observation that the inequality is not trivially satisfied because the right-hand size is not $\infty$, due to Lemma \ref{lem_gamma_bounds} stating that $\linf{n}\gengamma{n}>0$.
\end{proof}

\end{document}